\documentclass[en,oneside,bbfont,article]{amsart} 
\usepackage[lmargin = 30mm, rmargin = 30mm, bmargin = 25mm, tmargin=25mm]{geometry}
\usepackage[utf8]{inputenc}
\usepackage{xcolor}

\usepackage{mathrsfs}
\usepackage{mathtools}
\usepackage{enumitem}
\usepackage{mathabx}
\usepackage{mdframed}
\usepackage{amssymb}
\usepackage{dsfont}
\usepackage{amsthm}
\usepackage{bm}
\usepackage{amsmath,stackrel}
\usepackage{comment}
\usepackage{marginnote}
\usepackage{caption,subcaption}
\usepackage{scalefnt}
\usepackage{float,graphicx,verbatim}
\usepackage{tikz}
\usetikzlibrary{calc}
\usepackage{pgfplots}
\pgfplotsset{compat=1.16}
\RequirePackage[colorlinks=true,linkcolor=black,
	citecolor=blue,urlcolor=blue]{hyperref}

\usepackage{mathtools}
\mathtoolsset{showonlyrefs}
\numberwithin{equation}{section}

\providecommand{\examplename}{Example}

\newtheorem{theorem}{Theorem}[section]

\newtheorem{corollary}[theorem]{Corollary}
\newtheorem{lemma}[theorem]{Lemma}
\theoremstyle{remark}
\newtheorem{remarkx}[theorem]{Remark}
\newenvironment{remark}
    {\pushQED{\qed}\remarkx}
    {\popQED\endremarkx}
\theoremstyle{definition}
\newtheorem*{example*}{\protect\examplename}

\theoremstyle{plain}

\newtheorem*{assumption*}{Assumption}

\usepackage{comment}

\newcommand\N{\mathbb{N}}
\newcommand\Z{\mathbb{Z}}

\newcommand\R{\mathbb{R}}

\newcommand\E{\mathds{E}}
\newcommand\p{\mathds{P}}
\newcommand\1{\mathds{1}}
\newcommand\ld{,\ldots,}
\newcommand\eqd{\overset{\mathscr{D}}{=}}

\newcommand{\D}{\mathrm{d}}
\newcommand\cid{\xrightarrow{\mathscr{D}}}

\newcommand{\wt}[1]{\widetilde{#1}}

\newcommand\var{\mathrm{Var}}
\newcommand\cov{\mathrm{Cov}}

\newcommand{\CD}{\mathscr{C}}

\newcommand\mH{\mathfrak{H}}

\newcommand{\mD}{\mathbb{D}^{1,2}}

\usepackage[backend=biber,style=ieee,citestyle=numeric,sorting=nyt,maxbibnames=99,dashed=false,sortcites=true]{biblatex}
\renewbibmacro{in:}{}
\appto{\bibsetup}{\sloppy}
\usepackage[foot]{amsaddr}

\title[Exponential dimensional dependence in generalised method of moments]{Exponential dimensional dependence for Gaussian approximation in generalised method of moments}
\author{Andreas Basse-O'Connor$^*$ \& David Kramer-Bang$^{\dag}$}

\address{$^{*\dag}$Department of Mathematics, Aarhus University, DK}

\email{$^*$ basse@math.au.dk}
\email{$^\dag$ bang@math.au.dk}

\begin{document}

\begin{abstract}
It is numerically well known that moment-based tests for Gaussianity and estimators become increasingly unreliable at higher moment orders; however, this phenomenon has lacked rigorous mathematical justification. In this work, we establish quantitative bounds for Hermite-based moment tests, with matching exponential upper and lower bounds. Our results show that, even under ideal conditions with i.i.d.\ standard normal data, the sample size must grow exponentially with the highest moment order $d$ used in the test. These bounds, derived under both the convex distance and the Kolmogorov--Smirnov distance, are applied to classical procedures, such as the Shenton--Bowman test.
\end{abstract}



\maketitle

\section{Introduction}\label{sec:intro}
The Method of Moments, introduced by Karl Pearson in 1894, estimates parameters $\bm{\theta} \in \mathbb{R}^k$ by equating theoretical and empirical moments of a random variable $W$ depending on $\bm{\theta}$. Solving the resulting system of equations yields tests and estimators that are often computationally tractable, particularly when likelihood methods are unavailable. Since the first $n$ moments, cumulants, and Hermite polynomials are in one-to-one correspondence under Gaussianity, they are frequently used interchangeably in classical tests, including the Shenton--Bowman test~\cite{MR381079} and related methods~\cite{AMENGUAL2024,MR2067685,DECLERCQ1998101,MR4814822}.

Despite their widespread use, it is frequently noted in the literature~\cite{Max_Welling,Lo2015HighMoments,DeClerk2022HigherOrder} that moment-based methods perform poorly at high moment orders. In particular, reliable estimation or testing involving moments of order $d$ requires significantly larger sample sizes as $d$ increases. While this phenomenon is well known numerically, it has largely lacked rigorous justification. In this work, we formalise this by showing an exponential dependence on $d$, by proving nearly matching upper and lower bounds (both exponential) on the sample size required to detect deviations from normality using Hermite-based moment tests. Specifically, Theorem~\ref{thm:iid_MOM} establishes that the required sample size grows as $d^{3/4}e^{d \cdot 3\log(2)/2}$ for the upper bound in the convex distance~\eqref{eq:hyper_rect_defn} and as $d^{-3/4}e^{d \cdot 3\log(2)/2}$ for the lower bound in the Kolmogorov--Smirnov distance.

Concerns about the reliability of higher-order moment tests have been raised in several studies. In~\cite{Lo2015HighMoments}, the authors extend the classical Jarque--Bera test to include up to the first $4k$ moments and show through simulations that higher-order variants exhibit increased variability and require larger samples to perform comparably to classical tests. Higher-order moments in financial time series are analysed in~\cite{DeClerk2022HigherOrder}, and it is found that their estimates become highly unstable in small samples. Similarly, in~\cite{Max_Welling}, it is noted that high-order cumulants are extremely sensitive to sampling variability and outliers, making them difficult to estimate reliably without large datasets.

We show in this paper that the exponential scaling arises even in the most idealised setting of i.i.d.\ standard normal data, revealing a fundamental statistical limitation. To support the generality of this phenomenon, similar exponential bounds are also shown for dependent variables~\cite{Basse-Bang_Quant}, indicating that high-order moment conditions are inherently costly across a wide class of models.

The analysis of this paper reveals that sample complexity is driven not by the number of moment conditions used, but by the highest moment order involved. We demonstrate that even sparse use of high-order moments, like a test based solely on the sixth moment, is statistically as expensive as a test involving all moments up to that order. This is done via a lower bound (see Lemma~\ref{lem:lower_bound_Hermite_var}), which is proven by a reduction to a one-dimensional problem, which isolates the exponential cost to the index of the highest moment included. Thus, sharper detection sensitivity through high-order moments inevitably comes with exponential sample size requirements. Getting such lower bounds uses the bounds from~\cite{MR2573557}, and relies on Malliavin calculus, introduced in Section~\ref{sec:preliminaries}.

These findings clarify the role of quantitative bounds in moment-based inference, particularly in the Generalised Method of Moments (GMM), also discussed in~\cite[Sec.~2.1]{Basse-Bang_Quant}. Let $G_1,\ldots,G_n$ be a stationary time series governed by a model $\{\p_\xi: \xi \in \Xi\}$ with true parameter $\xi_0$. When likelihood-based methods are infeasible, GMM uses functions $\Phi:\mathbb{R} \to \mathbb{R}^d$ such that $\mathbb{E}_{\xi_0}[\Phi(G_1)] = 0$ if and only if $\xi = \xi_0$. To test the null hypothesis $\mathrm{H}_0\colon \xi = \xi_0$ against the alternative $\mathrm{H}_1\colon \xi \neq \xi_0$, we consider the test statistic $\bm{S}_n \coloneqq n^{-1/2} \sum_{k=1}^n \Phi(G_k)$. Under suitable conditions, $\bm{S}_n$ is approximately distributed as $\bm{Z} \sim \mathcal{N}_d(0, \bm{\Sigma})$, where $\bm{\Sigma}$ is the covariance matrix of $\Phi(G_1)$. For a given significance level $\alpha_0>0$, one can construct a Borel-measurable acceptance region  $A_\alpha\subset\mathbb{R}^d$, where $\mathbb{P}(\mathbf{Z}\in A_\alpha)=\alpha_0$, and where we accept the null-hypothesis $H_0$, if $\mathbf{S}_n\in A_\alpha$, and reject otherwise.

Although asymptotic normality is often justified via the Central Limit Theorem (CLT), statistical inference ultimately depends on finite-sample approximations. In particular, the model is only accurate to the extent that the Gaussian approximation accurately reflects the actual distribution of the test statistic. It is therefore crucial to understand how the approximation error depends on the sample size $n$ and the dimension $d$ of the test statistic. The \emph{finite-sample error},
\begin{equation}\label{eq:finite_sample_error}
    \text{Error}_{\text{f.s.}}(A_\alpha)\coloneqq \big|\mathbb{P}_{\xi_0}(\mathbf{S}_n\in A_\alpha)-\alpha_0\big|,
\end{equation}
quantifies the deviation of the distribution of~$\bm{S}_n$ from its Gaussian limit $\bm{Z}$ in the relevant geometric acceptance regions. This error directly impacts the reliability of hypothesis testing and the construction of valid confidence intervals. The two most used types of acceptance regions are; hyper-rectangles $A_\alpha=\{\mathbf{x}\in\mathbb{R}^d:\ \mathbf{v}_l\le \mathbf{x}\le \mathbf{v}_u\}$ for thresholds $\mathbf{v}_l,\mathbf{v}_u\in\mathbb{R}^d$, and classic geometric examples including ellipsoids
$A_\alpha=\{\mathbf{x}:\ \mathbf{x}^\top \mathbf{W}\,\mathbf{x}\le r_\alpha\}$ ($\bm{W}$ symmetric positive definite matrix) and Euclidean balls $A_\alpha=\{\mathbf{x}:\ \|\mathbf{x}\|\le r_\alpha\}$, all satisfying $\mathbb{P}(\mathbf{Z}\in A_\alpha)=\alpha_0$. 

To control the finite-sample error in~\eqref{eq:finite_sample_error}, we use the convex distance~$d_\CD$, which provides a uniform bound $\text{Error\textsubscript{f.s.}} \leq d_\CD(\bm{S}_n, \bm{Z})$. Our main result shows that this bound grows exponentially with~$d$. To demonstrate the sharpness of our upper bound, we also derive almost matching exponential lower bounds on~$d_\CD(\bm{S}_n, \bm{Z})$ via the Kolmogorov--Smirnov distance~$d_\mathcal{K}(\bm{S}_n, \bm{Z})$, defined in~\eqref{eq:hyper_rect_defn}, in the i.i.d.\ Gaussian setting. This framework applies directly to classical tests such as the Shenton--Bowman~\cite{MR381079}, Baringhaus-Henze-Epps-Pulley~\cite{MR980849,MR725389}, and Kolmogorov--Smirnov~\cite{Kol-33,MR1483} tests, as further discussed in~\cite{Basse-Bang_Quant}.

\subsection{Method of Moments vs.\ Hermite Method of Moments}

As discussed in the introduction, a central application of the Method of Moments is testing for Gaussianity. If a random variable $X$ has finite moments, then
\begin{equation}\label{eq:prob_met_of_mom}
X \sim \mathcal{N}(0,1) \qquad \text{if and only if} \qquad \mathbb{E}[X^q] =
\begin{dcases}
0, & \text{if } q \text{ is odd},\\
(q-1)!!, & \text{if } q \text{ is even},
\end{dcases}
\quad \text{for all } q \in \mathbb{N},
\end{equation}
where $q!!$ denotes the double factorial. Hence, to test whether a sequence of observations $G_1, \ldots, G_n$ is standard Gaussian, one can construct a moment-based test using the first $d$ equations from~\eqref{eq:prob_met_of_mom}. This leads to the test statistic
\begin{equation}\label{eq:MOM_defn_S_n_standard} 
\bm{S}_n = \frac{1}{\sqrt{n}} \sum_{k=1}^n \left(G_k, \frac{G_k^2 - 1}{\sqrt{2}}, \frac{G_k^3}{\sqrt{15}}, \ldots, \frac{G_k^d - \mathbb{E}[G_k^d]}{\sqrt{(2d - 1)!! - \mathbb{E}[G_k^d]^2}}\right)^\top, \quad \text{for all } n, d \in \mathbb{N}.
\end{equation}
Under the null hypothesis that $G_1, \ldots, G_n$ are i.i.d.\ standard normal, the vector $\bm{S}_n \in \mathbb{R}^d$ converges in distribution to a multivariate Gaussian $\bm{Z} \sim \mathcal{N}_d(\bm{0}, \bm{\Sigma})$ for some covariance matrix $\bm{\Sigma}$. A natural test statistic is $R_n = \|\bm{S}_n\|^2 \cid \|\bm{Z}\|^2$ as $n \to \infty$, where $\|\bm{x}\|^2 = \sum_{i=1}^d x_i^2$ for $\bm{x} \in \mathbb{R}^d$.

A major drawback of~\eqref{eq:MOM_defn_S_n_standard} is that the entries of $\bm{S}_n$ are not generally orthogonal, so the covariance matrix $\bm{\Sigma}$ is typically not the identity. This complicates analysis and may even cause the test to break down in practice, especially when $\bm{\Sigma}$ becomes singular. In some classical tests, such as the Shenton--Bowman test~\cite{MR381079,MR2067685}, only a subset of standardised moments (typically the third and fourth) is used, yielding orthogonal components.

In practical applications, it is important to design test statistics that avoid covariance singularities, since they degrade statistical performance. This motivates the use of orthogonal expansions, where the test components are constructed to be uncorrelated by design. A natural choice in the Gaussian setting is to use \emph{Hermite polynomials}, which form an orthogonal basis with respect to the standard Gaussian measure. In particular, for any random variable $X$ with finite moments,
\begin{equation}\label{eq:equivalence_normal_dist}
X \sim \mathcal{N}(0,1) \qquad \text{if and only if} \qquad \mathbb{E}[H_q(X)] = 0 \quad \text{for all } q \geq 1,
\end{equation}
cf.\ \cite[Thm~2.2.1]{MR2962301} and~\cite[Prop.~2]{MR1691731}. Here, $H_q$ denotes the $q$th Hermite polynomial, defined by
\[
H_q(x) \coloneqq (-1)^q e^{x^2/2} \frac{\D^q}{\D x^q} e^{-x^2/2}, \quad \text{for all } x \in \mathbb{R}, \; q \in \mathbb{N}.
\]
In particular, $H_0(x)\equiv 1$, $H_1(x) = x$ and $H_2(x) = x^2 - 1$. Since the family of Hermite polynomials $\{H_q : q \ge 0\}$ is an orthogonal basis of $L^2(\mathcal{N}(0,1))$~\cite[Prop.~1.4.2]{MR2962301}, condition~\eqref{eq:equivalence_normal_dist} offers a powerful tool for constructing normality tests based on orthogonal features. A test using the first $d$ equations of~\eqref{eq:equivalence_normal_dist} is referred to as a \emph{Hermite method of moments test}, or sometimes just \emph{Hermite test} (HT).

Under the null hypothesis that $G_1, \ldots, G_n$ are i.i.d.\ standard Gaussian, a Hermite-based test statistic can be defined using the vector $\bm{S}_n \in \mathbb{R}^d$, given by
\begin{equation}\label{eq:MOM_defn_S_n} 
\bm{S}_n \coloneqq \frac{1}{\sqrt{n}} \sum_{k=1}^n \left(H_1(G_k), \frac{H_2(G_k)}{\sqrt{2!}}, \ldots, \frac{H_d(G_k)}{\sqrt{d!}} \right)^\top,
\end{equation}
where each component is centered and normalised to have zero mean and unit variance. The norm $R^{HT}_n\coloneqq \|\bm{S}_n\|^2$ then serves as a test statistic for standard normality, and under the null-hypothesis, it converges in distribution to a $\chi^2(d)$ variable as $n \to \infty$ (see further discussion in Section~\ref{sec:satat_tests}).

The rest of the paper is structured as follows. The main quantitative bounds are stated in Section~\ref{sec:main_res}, where exact upper bounds are constructed for specific statistical tests, such as the Hermite polynomial statistical test, the Shenton--Bowman test, and the use of higher moments in statistical tests for normality are considered in Subsection~\ref{sec:satat_tests}. In Section~\ref{sec:preliminaries}, we introduce the most relevant results from Malliavin calculus needed for the proofs of the paper. All proofs are finally stated in Section~\ref{sec:proofs}.

\section{Main Result}\label{sec:main_res}
The main aim of this paper is to find a quantitative rate, completely explicit in $n$ and $d$, for $\bm{S}_n \cid \bm{Z}$ as $n \to \infty$, where $\bm{S}_n$ is the Hermite statistical test in~\eqref{eq:MOM_defn_S_n}. To find the rate of convergence, we need to decide on a metric to measure this convergence. 
We will show both upper and lower bounds for the metrics: Kolmogorov--Smirnov distance $d_{\mathcal{K}}$, 
and convex distance $d_\CD$ on $\R^d$. These metrics can be considered as a generalisation of the Kolmogorov metric on $\R$. The metrics are defined as follows:
\begin{align}
\begin{aligned}\label{eq:hyper_rect_defn}
    d_{\mathcal{K}}(\bm{X},\bm{Y})&\coloneqq \sup_{A \in \mathcal{K}}|\p(\bm{X} \in A)-\p(\bm{Y} \in A)|, \, \text{ where }&&\mathcal{K}=\bigg\{\bigtimes_{i=1}^d (-\infty,,b_i): -\infty \le b_i \le \infty\bigg\},\\
    d_\CD(\bm{X},\bm{Y})&\coloneqq \sup_{A \in \CD}|\p(\bm{X} \in A)-\p(\bm{Y} \in A)|, \, \text{ where } &&\CD=\{A \subset \R^d: A\text{ is convex}\},
\end{aligned}
\end{align} for all random vectors $\bm{X}$ and $\bm{Y}$ in $\R^d$, where $\mathcal{K}$ is the set of all left quadrants of $\R^d$, and $\CD$ is the set of all convex subsets of $\R^d$. From the definitions in~\eqref{eq:hyper_rect_defn}, it follows directly that the metrics are ordered: $d_{\mathcal{K}} \le 
d_\CD$.

\begin{theorem}\label{thm:iid_MOM} 
Let $\bm{S}_n$ be given in~\eqref{eq:MOM_defn_S_n} where $(G_k)_{k \in \N}$ is an i.i.d.\ Gaussian sequence, and assume that $\bm{Z}\sim \mathcal{N}_d(\bm{0},\bm{I}_d)$. Let $M \in \{\mathcal{K},\CD\}$. Then the following statements hold. 

\noindent\underline{Upper bound:} For $C \coloneqq 58 e^{3\log(2)/2} /(e^{3\log(2)/2}-1)$, it holds for all $n,d \in \N$, that 
     \begin{equation}\label{eq:up_low_bound_iid_MOM-1}
  d_M(\bm{S}_n,\bm{Z})  \le  C  d^{3/4} e^{d\cdot 3\log(2)/2}\, n^{-1/2}.
    \end{equation}
    
\noindent \underline{Lower bound:} Let $c \coloneqq (e\pi^32^3)^{-1/4}e^{-3\log(2)/2}$. For all $d\ge 2$ there exists a sequence $(N_d)_{d\ge 2}$ in $\N$, such that for all $n \ge N_d$,
\begin{equation}\label{eq:up_low_bound_iid_MOM}
     d_{M}(\bm{S}_n,\bm{Z}) \ge cd^{-3/4}e^{d \cdot 3\log(2)/2}\, n^{-1/2}.
    \end{equation}
\end{theorem}
Theorem~\ref{thm:iid_MOM} yields an exponential dependence on dimension for the Hermite method of moments in both the upper and lower bound~\eqref{eq:up_low_bound_iid_MOM}. But why does the Hermite method of moments depend so heavily on $d$? Is it the \emph{curse of dimensionality} or the large moments which cause this breakdown? The breakdown is not caused by the high dimensionality of $\bm{S}_n$, since the proof of the lower bound~\eqref{eq:up_low_bound_iid_MOM} only relies on the Kolmogorov distance between the last coordinate of $\bm{S}_n$ and $\bm{Z}$, making it a one-dimensional problem. Thus, the reason for an exponential $d$-dependence is the large moments.

\begin{remark}\label{rem:size_of_N_d}
\noindent (I) An important question in Theorem~\ref{thm:iid_MOM}, is how big $N_d$ should at least be? Since we have the trivial bound $d_{\mathcal{K}}(\bm{S}_n,\bm{Z}) \le 1$, we must have that $N_d \ge e^{3(d-1) \log(2)-3\log(d)/2}(e\pi^32^3)^{-1/2}$ for $d \ge 2$, for the inequality in~\eqref{eq:up_low_bound_iid_MOM} to be valid. 

\noindent (II) An upper bound of $d_\CD(\bm{S}_n,\bm{Z})$ for Gaussian sequences with a general dependence structure (i.e.\ where $(G_k)_{k \in \N}$ is not i.i.d.) can been found in~\cite{Basse-Bang_Quant}. The upper bound remains exponential in dimension $d$, with $d$-dependence given by $d^{125/24} e^{\log(3) d}$, and dependence on $n$ given by $n^{-1/2} \|\rho_n\|_{\ell^1(\Z)}^{3/2}$. 
\end{remark}

The proof of the lower bound of Theorem~\ref{thm:iid_MOM} is based on~\cite[Thm~3.1 \& Prop.~3.6]{MR2573557}, and relies on Malliavin calculus, introduced in Section~\ref{sec:preliminaries} below. The proof of the upper bound is based on~\cite[Thm~1.1]{MR2144310} and hypercontractivity of Hermite polynomials~\cite[Prop.~2.8.14]{MR2962301}.

\begin{remark}\label{sec:simulations}
    In this remark, we look at some simulations which visually support the results of Theorem~\ref{thm:iid_MOM}. By~\eqref{eq:KS_lower_bound}, it follows that $d_{\mathcal{K}}(\bm{S}_n,\bm{Z}) \ge d_{\mathcal{K}}(S_{n,d},Z_d)$, where $S_{n,d}$ (resp. $Z_d$) is the last entry in the vector $\bm{S}_n$ (resp. $\bm{Z}$). Hence, due to this and Remark~\ref{rem:size_of_N_d}, we can simulate the one-dimensional quantity $d_{\mathcal{K}}(S_{n,d},Z_d)$ for $n = 100{,}000$ and check if the finite size behaviour follows an exponential function as a function of $d$. We simulate for $d=2,\ldots,8$, which are plotted in Figure~\ref{fig:Simulations_MOM} below. Moreover, in Figure~\ref{fig:Simulations_MOM} we also fit function $d \mapsto a d^b e^{c d}$ for some $a,c>0$ and $b \in \R$. It is suggestive that the power-exponential function is a good fit. With only $7$ values, we get a fitted curve $d \mapsto 0.0015 d^{-2.19} e^{1.02 d}$, which asymptotically in $d$ is close to the lower bound found in Theorem~\ref{thm:iid_MOM}, namely $d \mapsto e^{d\cdot 3\log(2)/2}/(\sqrt{100{,}000}\, d^{3/4}(e\pi^32^3)^{1/4}e^{3\log(2)/2})$.

\begin{figure}[ht]
    \centering
    \centering\includegraphics[width=0.55\linewidth]{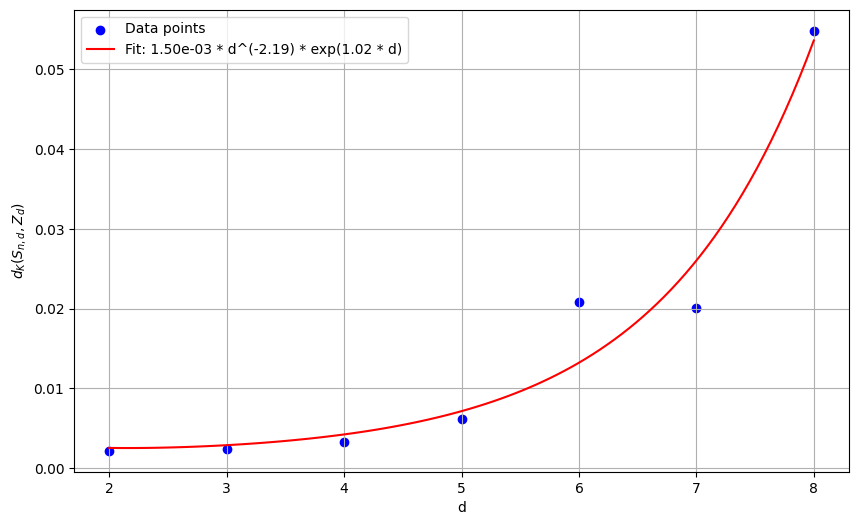}
    \caption{\small Plot of $d \mapsto  d_{\mathcal{K}}(S_{n,d},Z_d)$ for the values $d=2\ld 8$ and $n= 100{,}000$, together with the best fitted curve $d \mapsto a d^b e^{c d}$.}
    \label{fig:Simulations_MOM}
\end{figure}
\end{remark}

\subsection{Hermite Polynomials in Statistical Tests}\label{sec:satat_tests}
Below, we discuss examples from the papers~\cite{MR381079,MR2067685,MR4814822}, and how they relate to the setting and results of this paper. The examples will be the Shenton--Bowman test, the general use of Hermite polynomials in a statistical test, and the use of higher moments in statistical tests for normality. As a part of these examples, we show an exact rate of convergence for the Shenton--Bowman test and higher-moment statistical tests.

We start by considering the general use of Hermite polynomials in statistical tests for normality, also called a Hermite Test (HT), as suggested in the paper~\cite[Sec.~3.2]{MR2067685}. For $d \ge 1$ and $\bm{S}_n$ as in~\eqref{eq:MOM_defn_S_n}, we have that 
$\bm{S}_n \cid \bm{Z}$ as $n \to \infty$ where $\bm{Z} \sim \mathcal{N}_d(\bm{0},\bm{I}_d)$. Since $\bm{Z}$ has identity covariance matrix
, we know that the test statistic $R^{\mathrm{HT}}_n$ (see~\cite[Sec.~3.2]{MR2067685}) defined in~\eqref{eq:stat_test_hermite}, which builds on the Hermite polynomials $H_1, H_3, \ldots, H_{d}$ (see~\cite[Eq.~(3.8)]{MR2067685}), satisfies
\begin{equation}\label{eq:stat_test_hermite}
    R_n^{\mathrm{HT}} \coloneqq \sum_{q=1}^{d}\left(\frac{1}{\sqrt{n}} \sum_{k=1}^n \frac{H_q\left(G_k\right)}{\sqrt{q!}}\right)^2 \cid  R \sim \chi^2(d), \quad \text{ as }n \to \infty.
\end{equation} The test statistic $R_n^{\mathrm{HT}}$ in~\eqref{eq:stat_test_hermite} is based on the $d$ first Hernmite polynomials and motivated by~\eqref{eq:equivalence_normal_dist}, i.e. a statistical test for normality. In Corollary~\ref{cor:HermiteTest}, we are interested in finding the rate at which $d_{\mathcal{K}}(R_n^{\mathrm{HT}},R) \to 0$ as $n \to \infty$ 
and how this rate depends on $d$. 
\begin{remark}\label{rem:natural_d_C_conv}
Let $d \ge 1$, and consider a statistical test $R_n\coloneqq \|\bm{S}_n\|^2$, where $\bm{S}_n \in \R^d$ such that $\bm{S}_n \cid \bm{Z}$ as $n \to \infty$ with $\bm{Z}\sim \mathcal{N}_d(\bm{0},\bm{\Sigma})$ and $R\coloneqq \|\bm{Z}\|^2$. 
Then, it follows that 
\begin{equation*}
    d_{\mathcal{K}}(R_n,R) \le d_{\CD}(\bm{S}_n,\bm{Z}), \qquad \text{ for all }n \in \N. 
\end{equation*}
Using $d_{\CD}(\bm{S}_n,\bm{Z})$ is a natural upper bound, since considering $d_{\mathcal{K}}(R_n,R)=d_\mathcal{B}(\bm{S}_n,\bm{Z})$, where $\mathcal{B}\coloneqq \{B_{\bm{0}}(a):a \ge 0\} \subset \CD$, where $B_{\bm{0}}(a)\coloneqq \{\bm{x} \in \R^d: \|\bm{x}\| \le a \}$. Hence, since all sets in $\mathcal{B}$ are convex, it is natural to consider the upper bound $d_\mathcal{B}(\bm{S}_n,\bm{Z})\le d_\CD(\bm{S}_n,\bm{Z})$, concluding the statement. 
\end{remark}

\begin{corollary}\label{cor:HermiteTest}
    Let $d \ge 1$, $\bm{S}_n$ be defined in~\eqref{eq:MOM_defn_S_n} and let $\bm{Z} \sim \mathcal{N}_d(\bm{0},\bm{I}_d)$. Then, for $R_n^{\mathrm{HT}}$ and $R$ given in~\eqref{eq:stat_test_hermite}, and $C=58 e^{3\log(2)/2}/(e^{3\log(2)/2}-1)$, it follows that
    \begin{equation*}
        d_\mathcal{K}(R_n^{\mathrm{HT}},R) \le d_{\CD}(\bm{S}_n,\bm{Z}) \le C d^{3/4}e^{d \cdot  3\log(2)/2}\, n^{-1/2}, \quad \text{ for all } n,d \in \N.
    \end{equation*}
\end{corollary}
The proof of Corollary~\ref{cor:HermiteTest} is in Section~\ref{sec:proofs} below, and is based solely on Remark~\ref{rem:natural_d_C_conv} and Theorem~\ref{thm:iid_MOM}. Note by Theorem~\ref{thm:iid_MOM}, that a lower bound for $d_\CD(\bm{S}_n,\bm{Z})$ exists. Indeed, for all $d\ge 2$ there exists a sequence $(N_d)_{d\ge 2}$ in $\N$, such that $c d^{-3/4}e^{d\cdot 3\log(2)/2}n^{-1/2} \le d_\CD(\bm{S}_n,\bm{Z})$. 

\subsection{Classical statistical tests}
The main idea of translating statistical tests into their Hermite expansions and then applying Theorem~\ref{thm:iid_MOM} yields a method that works well for classical statistical tests, which we study in this subsection. Below, we introduce the classical Shenton--Bowman test statistic and higher-order moment test statistics.

The Shenton--Bowman test relies on the skewness and kurtosis, and has a strong connection to the 3rd and 4th Hermite polynomials~\cite{MR381079,MR2067685}. The Shenton--Bowman test statistic relies on the fact, that if $X \sim \mathcal{N}(0,1)$, then $\E[X^3]=0$ and $\E[X^4]-3=0$. Under the null-hypothesis, we observe i.i.d.\ $\mathcal{N}(0,1)$ samples $G_1, G_2, \ldots, G_n$. Then, by~\cite[Eq.~(3.14)]{MR2067685}, it follows that 
\begin{equation}\label{eq:Shent_Bow_test}
    \wt{\bm{S}}_n \coloneqq \frac{1}{\sqrt{n}} \sum_{k=1}^n\left(G_k^3/\sqrt{15},(G_k^4-3)/\sqrt{96}\right)^\top \cid  \mathcal{N}_2\left(\bm{0},\bm{I}_2\right) , \quad  \text{ as }n \to \infty.
\end{equation} The Shenton--Bowman test statistic is given by 
\begin{equation}\label{eq_Shenton--Bowman}
    R_n^{\mathrm{SB}}\coloneqq \|\wt{\bm{S}}_n\|^2=\bigg(\frac{1}{\sqrt{n}} \sum_{k=1}^n G_k^3/\sqrt{15}\bigg)^2+\bigg(\frac{1}{\sqrt{n}} \sum_{k=1}^n\left(G_k^4-3\right)/\sqrt{96}\bigg)^2 \cid R \sim \chi^2(2), \, \text{ as }n \to \infty.
\end{equation}
We are interested in finding a rate of convergence for the Shenton--Bowman test statistic, i.e. the rate at which $d_\mathcal{K}(R_n^{\mathrm{SB}},R) \to 0$ as $n \to \infty$ 
(see Corollary~\ref{cor:shenton_bowman} below).

A similar statistical test as the Shenton--Bowman test can be constructed for the higher-order moment test~\cite{MR4814822}, which use up to the 6th moment. The $HM4$ statistical test in~\cite[p.~4782]{MR4814822} translated to our notation with known mean and variance, i.e. no self-normalisation, is given by
\begin{align}\label{eq:HM4_stat_test_R}
\begin{aligned}
    R_n^{\mathrm{HM4}} &\coloneqq \bigg(\frac{1}{\sqrt{n}} \sum_{k=1}^n \frac{G_k^3}{\sqrt{15}}\bigg)^2+\bigg(\frac{1}{\sqrt{n}} \sum_{k=1}^n\frac{G_k^4-3}{\sqrt{96}}\bigg)^2 + \left( \frac{1}{\sqrt{n}} \sum_{k=1}^n \frac{G_k^5-10 G_k^3}{\sqrt{345}}\right)^2\\
    &\quad +\left( \frac{1}{\sqrt{n}} \sum_{k=1}^n \frac{G_k^6-15G_k^4+30}{\sqrt{4770}}\right)^2
    \cid \wt Z_1^2+\wt Z_2^2+\wt Z_3^2+\wt Z_4^2, \text{ as }n \to \infty,
\end{aligned}
\end{align} where $\wt{\bm{Z}}=(\wt Z_1\ld \wt Z_4)^\top \sim \mathcal{N}_4(\bm{0},\wt{\bm{\Sigma}})$, and $\wt{\bm{\Sigma}}
$ is given by
\begin{equation}\label{eq:HM4_cov}
    \wt{\bm{\Sigma}} = 
   \begin{pmatrix}
       1 & 0 & -\frac{3 \sqrt{23}}{23} & 0  \\ 
       0 & 1 
       & -\frac{3\sqrt{115}}{46} & 0
        \\
       -\frac{3 \sqrt{23}}{23} & -\frac{3\sqrt{115}}{46} & 1 & \frac{45 \sqrt{3657}}{4876}  \\
       0 & 0 & \frac{45 \sqrt{3657}}{4876} & 1
   \end{pmatrix}.
\end{equation}
As in the Shenton--Bowman case, we are interested in the rate of convergence for $d_{\mathcal{K}}(R_n^{\mathrm{HM4}},\|\wt{\bm{Z}}\|^2)\to 0$ as $n \to \infty$ (see Corollary~\ref{cor:shenton_bowman} below). Note that $\wt{\bm{\Sigma}} \ne \bm{I}_d$, and therefore $\|\wt{\bm{Z}}\|^2$ is not $\chi^2(4)$-distributed.

As discussed in Remark~\ref{rem:natural_d_C_conv}, it is natural in both the case of $R_n^{\mathrm{SB}}$ and $R_n^{\mathrm{HM4}}$, to study the rate of convergence through Theorem~\ref{thm:iid_MOM}, and hence through $d_\CD(\bm{S}_n,\bm{Z})$ where $\bm{S}_n$ is given in~\eqref{eq:MOM_defn_S_n}, and $\bm{Z} \sim \mathcal{N}_d(\bm{0},\bm{I}_d)$ for some appropriate $d$. 

\begin{corollary}\label{cor:shenton_bowman}
Let $\bm{S}_n$ be given as in~\eqref{eq:MOM_defn_S_n}, and assume that $(G_k)_{k \in \N}$ is an i.i.d.\ Gaussian sequence. Then the following statements hold.

    \noindent{(a)} Let $d=4$, and assume that $R_n^{\mathrm{SB}}$ is the Shenton--Bowman test statistic given in~\eqref{eq_Shenton--Bowman}, and $R \sim \chi^2(2)$. Then for $\bm{Z} \sim \mathcal{N}_4(\bm{0},\bm{I}_4)$, it follows that
    \begin{equation*}
        d_\mathcal{K}\big(R_n^{\mathrm{SB}},R\big) \le d_\CD(\bm{S}_n,\bm{Z}) \le  924 \, n^{-1/2}, \quad \text{ for all }n \in \N.
    \end{equation*}

    \noindent{(b)} Let $d=6$, and assume that $R_n^{\mathrm{HM4}}$ is the statistical test from~\eqref{eq:HM4_stat_test_R} with $\wt{\bm{Z}} \sim \mathcal{N}_4(\bm{0},\wt{\bm{\Sigma}})$ and $\wt{\bm{\Sigma}}$ given in~\eqref{eq:HM4_cov}. Then, for $\bm{Z} \sim \mathcal{N}_6(\bm{0},\bm{I}_6)$, it follows that
    \begin{equation*}
        d_\mathcal{K}\big(R_n^{\mathrm{HM4}},\|\wt{\bm{Z}}\|\big) \le d_\CD(\bm{S}_n,\bm{Z}) \le  673{,}795\, n^{-1/2}, \quad \text{ for all }n \in \N.
    \end{equation*} 
\end{corollary}

Note that the proof of Corollary~\ref{cor:shenton_bowman} is based on Remarks~\ref{rem:natural_d_C_conv} \&~\ref{rem:alt_up_bound_d} together with Theorem~\ref{thm:iid_MOM}, and is stated in Section~\ref{sec:proofs} below. The main idea is to translate the setting from the standardised moments to Hermite polynomials and then use Theorem~\ref{thm:iid_MOM} together with Remark~\ref{rem:alt_up_bound_d}, yielding an exact upper bound.

The key conclusion of Corollary~\ref{cor:shenton_bowman} is that the constant $924$ associated with the statistic $R^{\mathrm{SB}}_n$ is significantly smaller than the corresponding constant $ 673{,}795$ for $R^{\mathrm{HM4}}_n$. This aligns with the theoretical expectation that the statistical cost grows exponentially with the order of the highest moment involved in the test. Although the test statistic $R^{\mathrm{HM4}}_n$ may be more sensitive to fine-grained deviations from the null distribution, its reliance on a higher-order moment (specifically the 6th) incurs a substantially greater sample complexity. In practical terms, this means that $R^{\mathrm{HM4}}_n$ requires significantly more data to achieve reliability comparable to $R^{\mathrm{SB}}_n$.

\begin{remark}
Additional statistical tests that rely on higher-order moments are discussed in~\cite[p.~4782]{MR4814822}. In particular, the authors examine the use of each component of the $HM4$ statistic as individual directional test statistics based on specific higher moments. These parametric directional tests, denoted $R^{\mathrm{M5}}_n$, $R^{\mathrm{M6}}_n$, and $R^{\mathrm{HM2}}_n$ (in the absence of self-normalisation), are, as $n \to \infty$, given by
\begin{gather}
R_n^{\mathrm{M5}}  \coloneqq \left( \frac{1}{\sqrt{n}} \sum_{k=1}^n \frac{G_k^5-10 G_k^3}{\sqrt{345}}\right)^2 \cid \chi^2(1), \\
R_n^{\mathrm{M6}} \coloneqq \left( \frac{1}{\sqrt{n}} \sum_{k=1}^n \frac{G_k^6-15G_k^4+30}{\sqrt{4770}}\right)^2 \cid \chi^2(1), \, \text{ and } \, R_n^{\mathrm{HM2}} \coloneqq  R_n^{\mathrm{M5}}+R_n^{\mathrm{M6}}\cid \chi^2(2).
\end{gather} Note that $R_n^{\mathrm{HM2}} \cid \chi^2(2)$ as $n \to \infty$, since the polynomials $H_5(G)-15H_1(G)$ and $ H_6(G)-45H_2(G)$ are orthogonal in $L^2(\Omega)$ for $G \sim \mathcal{N}(0,1)$. The same approach as in Corollary~\ref{cor:shenton_bowman} can be used to construct a convergence rate for the statistical tests $R^{\mathrm{M5}}_n$, $R^{\mathrm{M6}}_n$ and $R^{\mathrm{HM2}}_n$. Due to the exponential dependence on the order of the highest moment, the rate of convergence is rather slow, which is also discussed in~\cite[p.~4782]{MR4814822}, where the authors state ``Being based on statistics with slow rates of convergence to asymptotic normal values...''. However, no theoretical argument is given for why this is the case. Nevertheless, this can be seen from the fact that the rate depends exponentially on the magnitude of the size of the highest moment in the statistical test, as shown in this example.
\end{remark}

\section{Preliminaries}\label{sec:preliminaries}
In this section, we introduce the most relevant results from Malliavin calculus needed in the paper. For a full monograph on Malliavin calculus, see~\cite{MR2200233}. Let $\mH$ be a real separable Hilbert space with inner product $\langle \cdot,\cdot\rangle_\mH$ and norm $\|\cdot\|_\mH=\langle \cdot,\cdot\rangle_\mH^{1/2}$. We call $X=\{X(h):h \in \mH\}$ an \emph{isormormal Gaussian process} over $\mH$, where $X$ is defined on some probability space $(\Omega,\mathcal{F},\mathds{P})$, whenever $X$ is a centered Gaussian family indexed by $\mH$ such that $\E[X(h)X(g)]=\langle h,g\rangle_{\mH}$. Define $\mathcal{F}$ to be the $\sigma$-algebra generated by $X$, i.e. $\mathcal{F}=\sigma\{X\}$, and write $L^2(\Omega,\mathcal{F},\mathds{P})=L^2(\Omega)$, where we say that $F \in L^2(\Omega)$ if $\E[|F|^2]<\infty$. The \emph{Wiener-It\^o chaos expansion} (see~\cite[Cor.~2.7.8]{MR2962301}), states, for any $F \in L^2(\Omega)$, that $F=\E[F]+\sum_{q=1}^\infty I_q(f_q)$ where $I_0(f_0)=\E[F]$, $I_q(f_q)$ is the $q$'th multiple Wiener-It\^o integral and the kernels $f_q$ are in the symmetric $q$'th tensor product $ \mH^{\odot q}$. For an $F \in L^2(\Omega)$, we write $F \in \mD$ if $\sum_{q=1}^\infty q!q\|f_q\|_{\mH^{\otimes q}}^2<\infty$.
,
and for $F \in \mD$ we let the random element $D F$ with values in $\mH$ be the \emph{Malliavin derivative}. Moreover, such $F$ satisfies the \emph{chain rule}, indeed, let $\bm{F}=(F_1\ld F_d)$ where $F_i \in \mD$ for $i=1\ld d$ and let $\phi:\R^d \to \R$ be a continuous differentiable function with bounded partial derivatives. Then, by~\cite[Prop.~1.2.3]{MR2200233}, $\phi(\bm{F}) \in \mD$ and $D\phi(\bm{F})=\sum_{i=1}^d \big( \tfrac{\partial}{\partial x_i} \phi(\bm{F})\big)\cdot D F_i$.

For $(G_k)_{k \in \N}$ being an i.i.d.\ sequence of $\mathcal{N}(0,1)$ variables, we have that $\rho(j-k)=\E[G_jG_k]=0$ if $i \ne j$ and $\rho(0)=1$. By~\cite[Rem.~2.1.9]{MR2962301} we may choose an isonormal Gaussian process $\{X(h) \in \mH\}$, such that $(G_k)_{k \in \N} \eqd \{X(e_k):k \in \N\}$ where $(e_k)_{k \in \N}\subset \mH$ verifying $\langle e_k,e_j\rangle_\mH = \rho(j-k)$ for all $j,k \in \N$.  It then follows by~\cite[Thm~2.7.7]{MR2962301} that $H_q(G_r)\eqd I_q(e^{\otimes q}_r)$ for all $q \ge 1$ and $r \in \N$. Since $D(G_r)=e_r$ by~\cite[Rem.~2.9.2]{MR2962301}, and $H_q'(x)=q H_{q-1}(x)$ by~\cite[Prop.~1.4.2(i)]{MR2962301}, the chain-rule implies that $DH_q(G_r)=qH_{q-1}(G_r)e_r$ for all $q\ge 2$ and $r \in \N$. Finally, the $r$th contraction operator $\otimes_r$ is defined in~\cite[App.~B.4]{MR2962301}.

\section{Proofs}\label{sec:proofs} 

\subsection{Proof of Theorem~\ref{thm:iid_MOM}}
To prove the lower bound in Theorem~\ref{thm:iid_MOM}, the ensuing lemma is needed. 
\begin{lemma}\label{lem:lower_bound_Hermite_var} 
Assume that $\bm{S}_n$ is given by~\eqref{eq:MOM_defn_S_n}, that $(G_k)_{k \in \N}$ is an i.i.d.\ Gaussian sequence and that $\bm{Z} \sim \mathcal{N}(\bm{0},\bm{I}_d)$. For all $d \ge 2$, there exists some $N_d \in \N$, such that, for all $n \ge N_d$,
    \begin{align*}
        d_{\mathcal{K}}(\bm{S}_n,\bm{Z}) \ge \frac{1}{\sqrt{n}(e\pi^32^3)^{1/4}} \begin{dcases}
            d^{-d/4}e^{3d\log(2)/2}, & \text{ if }d \text{ is even},\\
            (d-1)^{-3/4}e^{3(d-1)\log(2)/2}, & \text{ if }d \text{ is odd}.
        \end{dcases}
    \end{align*}
\end{lemma}
Throughout the paper, we will often apply Stirling's inequality~\cite[Eqs.~(1) \&~(2)]{MR69328}, given by
\begin{equation}\label{eq:Stirling's_formula}
    \sqrt{2\pi n}\left( \frac{n}{e}\right)^n 
    <
      n!
      <
      e^{\frac{1}{12}}\sqrt{2\pi n}\left( \frac{n}{e}\right)^n, \quad \text{ for all }n \in \N.
\end{equation} 
\begin{proof}
    Let $d \ge 2$, and recall that $d_{\mathcal{K}}$ is a weaker metric than $d_\CD$. Hence, 
    \begin{equation}\label{eq:KS_lower_bound}
        d_\CD(\bm{S}_n,\bm{Z}) \ge d_{\mathcal{K}}(\bm{S}_n,\bm{Z}) \ge \begin{dcases}
            d_{\mathcal{K}}(S_{n,d},Z_d), & \text{ if } d \text{ is even},\\
            d_{\mathcal{K}}(S_{n,d-1},Z_{d-1}), & \text{ if } d \text{ is odd}.
        \end{dcases}
    \end{equation} Here $S_{n,d}$ (resp. $S_{n,d-1}$) is the last (resp. penultimate) element of $\bm{S}_n$ and likewise $Z_d$ (resp. $Z_{d-1}$) is the last (resp. penultimate) element in $\bm{Z}$. Since the argument is now reduced to a one-dimensional problem, we will assume without loss of generality that $d$ is even and fixed. If $d$ was odd, we should just consider $d-1$, which in this case is even. 
    
    It thus suffices to find a lower bound in the one-dimensional setting $d_{\mathcal{K}}(S_{n,d},Z_d)$. We will use~\cite[Thm~3.1 \& Prop.~3.6]{MR2573557} to construct such a lower bound. Since $\rho(v)=0$ for all $v\in \Z \setminus\{0\}$ and by~\cite[Prop.~2.2.1]{MR2962301}, it follows that $\var(S_{n,d})=(d! n)^{-1} \sum_{k,r=1}^n \E[H_d(G_k)H_d(G_r)] = n^{-1}\sum_{k,r=1}^n \rho(k-r)=1$. We now verify the assumptions in~\cite[Prop.~3.6]{MR2573557}. First, we note that $S_{n,d}$ is indeed in a fixed Wiener chaos, since $S_{n,d}\eqd I_d\big((d!n)^{-1/2}\sum_{k=1}^n e_k^{\otimes d}\big)\coloneqq I_d(f(n,d))$ for $f(n,d) \in \mH^{\odot d}$. To verify~\cite[Prop.~3.6]{MR2573557}, we start by calculating $\Psi(n)\coloneqq \E[(1-d^{-1} \| D S_{n,d}\|_{\mH}^2)^2]^{1/2}$, and to do so, we note that
    \begin{equation*}
        \| D S_{n,d}\|_{\mH}^2
        =\frac{d}{(d-1)! n}\sum_{k,r=1}^nH_{d-1}(G_k)H_{d-1}(G_r)\rho(k-r)=\frac{d}{(d-1)! n}\sum_{k=1}^nH_{d-1}(G_k)^2.
    \end{equation*} Hence, by definition of $\Psi(n)$, using linearity of expectations and~\cite[Prop.~ 2.2.1]{MR2962301}, it follows that
    \begin{align*}
        \Psi(n)^2
        &=
        1-\frac{2}{(d-1)!n}\sum_{k=1}^n \E[H_{d-1}(G_k)^2]+\frac{1}{((d-1)!)^2n^2}\sum_{k,r=1}^d \E[H_{d-1}(G_k)^2H_{d-1}(G_r)^2]\\
        &=
        -1+\frac{1}{((d-1)!)^2n^2} \Bigg( \sum_{\underset{k \ne r}{k,r \in \{1\ld n\}:}}\E[H_{d-1}(G_k)^2] \E[H_{d-1}(G_r)^2]+\sum_{k=1}^n \E[H_{d-1}(G_k)^4]\Bigg)\\
        &=
        -1+\frac{\left( n(n-1)((d-1)!)^2+n \E[H_{d-1}(G_1)^4]\right)}{((d-1)!)^2n^2} =\frac{1}{n}\left( \frac{\E[H_{d-1}(G_1)^4]}{((d-1)!)^2}-1\right).
    \end{align*}
    We can now verify the assumptions of~\cite[Eq.~(3.6), (3.7) \&~(3.8)]{MR2573557}. Let $s \in \{1\ld d-1\}$, and note by definition of the kernels $f(n,d)$, that
    \begin{align}
        \| f(n,d) \otimes_s f(n,d)\|_{\mH^{\otimes(2d-2s)}}^2&=\sum_{k,k',r,r'=1}^n \frac{1}{(d!)^2n^2} \langle e_k^{\otimes d}\otimes_s e_{k'}^{\otimes d},e_{r}^{\otimes d}\otimes_s e_{r'}^{\otimes d}\rangle_{\mH^{\otimes(2d-2s)}}, \, \text{ where}\nonumber \\
        e_k^{\otimes d}\otimes_s e_{k'}^{\otimes d}&=\left( \prod_{\ell=1}^r \langle e_k,e_{k'}\rangle_{\mH}\right)e_k^{\otimes (d-s)}\otimes e_{k'}^{\otimes (d-s)} = \begin{dcases}
            0, &\text{ if }k \ne k',\\
            e_k^{\otimes (2d-2s)}, &\text{ if } k=k'.
        \end{dcases} \label{eq:otimes_s_general}
    \end{align} Thus, since $\langle e_k^{\otimes(2d-2s)},e_r^{\otimes(2d-2s)}\rangle_{\mH^{\otimes (2d-2s)}}=\prod_{\ell=1}^{2d-2s} \langle e_k,e_r \rangle_{\mH}=\1_{\{k=r\}}$,
    it follows that $\| f(n,d) \otimes_s f(n,d)\|_{\mH^{\otimes(2d-2s)}}^2=((d!)^2n)^{-1} \to 0$, as $n \to \infty$, verifying assumption~\cite[Eq.~(3.6)]{MR2573557}. It follows directly that $(1-d!\|f(d,n)\|_{\mH^{\otimes d}})/\Psi(n)=0$ for all $n\ge 2$ since $\var(S_{n,d})=1$, verifying~\cite[Eq.~(3.7)]{MR2573557}. Let $s \in \{1\ld  d-1\}$ and $\ell \in \{1\ld  2d-2s-1\}$. Note by construction and symmetry of $f(n,d)$, that
    \begin{equation*}
        g(n,d,s)\coloneqq f(n,d) \wt{\otimes}_s f(n,d) = \frac{1}{d!n} \sum_{k,r=1}^n e_k^{\otimes d} \otimes_s e_r^{\otimes d}=\frac{1}{d!n} \sum_{k=1}^n e_k^{\otimes (2d-2s)}.
    \end{equation*} By~\eqref{eq:otimes_s_general} and since $\rho(v)=0$ for all $v \in \Z\setminus\{0\}$, it follows that
   \begin{equation*}
        g(n,d,s) \otimes_\ell g(n,d,s)
       =
        \frac{1}{(d!)^2n^2}\sum_{k,r=1}^n e_k^{\otimes (2d-2s)}\otimes_\ell e_r^{\otimes (2d-2s)}
        =
        \frac{1}{(d!)^2n^2} \sum_{k=1}^n e_k^{\otimes (2(2d-2s)-2\ell)}.
    \end{equation*} Applying this, we see that
    \begin{align*}
        &\|g(n,d,s) \otimes_\ell g(n,d,s)\|_{\mH^{\otimes (2(2d-2s)-2\ell)}}\Psi(n)^{-2}\\
        &\qquad = \Psi(n)^{-2}\bigg(\frac{1}{(d!)^4n^4} \sum_{k,r=1}^n \langle e_k^{\otimes (2(2d-2s)-2\ell)},e_r^{\otimes (2(2d-2s)-2\ell)}\rangle_{\mH^{\otimes (2(2d-2s)-2\ell)}} \bigg)^{1/2}\\
        &\qquad = \frac{1}{(d!)^2 n^{3/2}\Psi(n)^{2}} \longrightarrow 0, \quad \text{ as } n \to \infty,
    \end{align*} verifying~\cite[Eq.~(3.8)]{MR2573557}. Finally, note by definition of $\Psi(n)$, that 
    \begin{align*}
        \zeta(d)&\coloneqq -dd!(d/2-1)!\binom{d-1}{d/2-1}^2 \Psi(n)^{-1}\langle f(n,d),f(n,d)\wt\otimes_{d/2}f(n,d)\rangle_{\mH^{\otimes d}}\\
        &=-d(d!)^{-1/2}(d/2-1)!\binom{d-1}{d/2-1}^2 \left( \frac{\E[H_{d-1}(G_1)^4]}{((d-1)!)^2}-1\right)^{-1}, 
    \end{align*} since $f(n,d)\wt\otimes_{d/2} f(n,d)=(d!n)^{-1}\sum_{k=1}^n e_k^{\otimes d}$, which implies that $\langle f(n,d),f(n,d)\wt\otimes_{d/2}f(n,d)\rangle_{\mH^{\otimes d}} = (d!)^{-3/2}n^{-1/2}$. Note that $\zeta(d)$ purely depends on $d$, and hence does not depend on $n$. 
    
    Applying~\cite[Thm~3.1 \& Prop.~3.6]{MR2573557}, it now follows that there exists some $N_d \in \N$ such that
    \begin{equation*}
        d_{\mathcal{K}}(S_{n,d},Z_d) \ge (3/\sqrt{2\pi})\Psi(n)|\zeta(d)|, \quad \text{ for all }n \ge N_d.
    \end{equation*} The proof can now be concluded, by using the definition of $\zeta(d)$ and multiple applications of Stirling's approximation~\eqref{eq:Stirling's_formula} to get a lower bound on $\Psi(n)|\zeta(d)|$:
    \begin{align*}
        \sqrt{n}\Psi(n)|\zeta(d)|&= \frac{d(d/2-1)! \binom{d-1}{d/2-1}^2}{\sqrt{d!}}=\frac{\sqrt{d!}(d-1)!}{(d/2-1)!((d/2)!)^2} \\
        &\ge
        \frac{1}{(2e\pi^3)^{1/4}}\left( \frac{d-2}{d-1}\right)^{1/2}\frac{(d-1)^d}{d^{d/2}(d-2)^{d/2}} \frac{2^{3d/2}}{d^{3/4}} \ge \frac{1}{(2^3 e \pi^3)^{1/4}} e^{3d\log(2)/2-3\log(d)/4}.\qedhere
    \end{align*} 
\end{proof}

\begin{proof}[Proof of Theorem~\ref{thm:iid_MOM}]
    Assume now that $(G_k)_{k \in \N}$ is an i.i.d.\ Gaussian sequence. We start by proving the lower bound in~\eqref{eq:up_low_bound_iid_MOM}. For all $d \ge 2$, Lemma~\ref{lem:lower_bound_Hermite_var} implies that there exists some $N_d \in \N$ such that it for all $n \ge N_d$ 
    \begin{align*}
        d_{\mathcal{K}}(\bm{S}_n,\bm{Z}) \ge \frac{1}{\sqrt{n}(e\pi^32^3)^{1/4}} \begin{dcases}
            e^{3d\log(2)/2-3\log(d)/4}, & \text{ if }d \text{ is even},\\
            e^{3(d-1)\log(2)/2-3\log(d-1)/4}, & \text{ if }d \text{ is odd}.
        \end{dcases}
    \end{align*}Thus, since $\log(d+1)\ge \log(d)$, it follows that
    \begin{equation}\label{eq:explicit_lower_bound}
        d_{\mathcal{K}}(\bm{S}_n,\bm{Z})\ge \frac{1}{(e\pi^32^3)^{1/4}e^{3\log(2)/2}} \frac{e^{3d\log(2)/2}}{d^{3/4}\sqrt{n}}, \qquad \text{ for all }n \ge N_d.
    \end{equation} 

    Next, we prove the upper bound in~\eqref{eq:up_low_bound_iid_MOM-1} for all $d \ge 1$. Note that $\cov(\bm{S}_n)=\cov(\bm{Z})=\bm{I}_d$, since the orthogonality of the Hermite polynomials yields $\cov(\bm{S}_n)_{i,j}=0$ if $i \ne j$ and moreover if $i=j$ the assumption of i.i.d.\ together with~\cite[Prop.~2.2.1]{MR2962301} yields
    \begin{equation*}
    \cov(\bm{S}_n)_{i,i}=\var\left(\frac{1}{\sqrt{n}}\sum_{k=1}^n \frac{H_{i}(G_k)}{\sqrt{i!}}\right)= \var\left( \frac{H_i(G_1)}{\sqrt{i!}}\right)=1.
    \end{equation*}
    Hence, by~\cite[Thm~1.1]{MR4003566} (see also~\cite[Thm~1.1]{MR2144310}) and Jensen's inequality, it follows that 
    \begin{equation}\label{eq:up_bound_w_HP}
       d_\CD(\bm{S}_n,\bm{Z})\le \frac{ (42d^{1/4}+16)}{\sqrt{n}} \E\Bigg[ \left( \sum_{i=1}^{d}\varphi_i(G_1)^2\right)^{3/2}\Bigg] \le \frac{d^{1/2}(42d^{1/4}+16)}{\sqrt{n}}\sum_{i=1}^d \E[|\varphi_i(G_1)|^3] ,
    \end{equation} where $\varphi_i(x)=H_{i}(x)/\sqrt{i!}$. Thus, it suffices to bound $\E[|\varphi_i(G_1)|^3]$ in terms of $i$. By~\cite[Prop.~2.2.1 \&~Cor.~2.8.14]{MR2962301}, it follows that $$\E[|H_i(G_1)|^3] \le 2^{3 i/2}\E[H_i(G_1)^2]^{3/2} = 2^{3i/2} (i!)^{3/2}, \quad \text{ for all }i\ge 1.$$
    Hence, $ \E[|\varphi_i(G_1)|^3]=
        \E[|H_{i}(G_1)|^3]/(i!)^{3/2} \le e^{i \cdot 3\log(2)/2}$, for $i=1 \ld d$. Using the closed form for partial geometric series, it follows, together with $d^{1/2}(42d^{1/4}+16) \le 58 d^{3/4}$ for all $d \ge 1$, that 
    \begin{equation}\label{eq:exact_upper_bound}
        d_\CD(\bm{S}_n,\bm{Z})
        \le 
        \frac{58 d^{3/4}}{\sqrt{n}} \sum_{i=1}^d e^{i(3\log(2)/2)} 
        \le 
        \frac{58 e^{3\log(2)/2}}{e^{3\log(2)/2}-1} \frac{d^{3/4} e^{d(3\log(2)/2)}}{\sqrt{n}}.\qedhere
    \end{equation} 
\end{proof}

\subsection{Proofs of Examples from Section~\ref{sec:satat_tests}}

\begin{proof}[Proof of Corollary~\ref{cor:HermiteTest}]
The main idea of the proof is to apply Theorem~\ref{thm:iid_MOM}, and thus find the rate of convergence for~\eqref{eq:stat_test_hermite}. As in Remark~\ref{rem:natural_d_C_conv}, it follows that $d_\mathcal{K}(R_n^{\mathrm{HT}},R) \le d_\CD(\bm{S}_n,\bm{Z})$, for all $n \in \N$, where $\bm{Z}=(Z_1\ldots Z_d)^\top \sim \mathcal{N}_{d}(\bm{0},\bm{I}_{d})$ such that $Z_1^2+\cdots +Z_{d}^2\eqd R \sim \chi^2(d)$. Next,~\eqref{eq:up_low_bound_iid_MOM-1} implies that  
\begin{equation*}
    d_\mathcal{K}(R_n^{\mathrm{HT}},R) 
    \le d_{\CD}(\bm{S}_n,\bm{Z}) 
    \le \frac{58 e^{3\log(2)/2}}{e^{3\log(2)/2}-1} d^{3/4}e^{d\cdot 3\log(2)/2}\, n^{-1/2}, \quad \text{ for all } n,d \in \N.\qedhere
\end{equation*}
\end{proof}

Before we continue to the proof of Corollary~\ref{cor:shenton_bowman}, we consider the following useful remark, which can be concluded from the proof of Theorem~\ref{thm:iid_MOM} above.

\begin{remark}\label{rem:alt_up_bound_d}
    Assume that $\bm{S}_n$ is as in~\eqref{eq:MOM_defn_S_n} and $\bm{Z}\sim \mathcal{N}_d(\bm{0},\bm{I}_d)$. It is possible to find an upper bound with exact constants and an explicit formula in $d$. However, this method is only tractable for smaller $d$, since it requires calculating the $4$th moment of the Hermite polynomials. Indeed, the bound~\eqref{eq:up_bound_w_HP} together with Jensen's inequality, yields a bound which numerically can be calculated for smaller $d$:
\begin{equation}\label{eq:exact_up_bound_small_d}
    d_{M}(\bm{S}_n,\bm{Z}) 
    \le \frac{d^{1/2}(42d^{1/4}+16)}{\sqrt{n}}\sum_{i=1}^d \frac{\E[|H_i(G_1)|^4]^{3/4} }{(i!)^{3/2}}\eqqcolon \frac{C_d}{\sqrt{n}}, \text{ for all }n,d \in \N.\qedhere
\end{equation} 
\end{remark}

\begin{proof}[Proof of Corollary~\ref{cor:shenton_bowman}]
Part~(a). Let $\wt{\bm{S}}_n$ be given by~\eqref{eq:Shent_Bow_test} with $R_n^{\mathrm{SB}}=\|\wt{\bm{S}}_n\|^2$. By Remark~\ref{rem:natural_d_C_conv}, it follows that $d_\mathcal{K}\big(R_n^{\mathrm{SB}},R\big) 
     = \sup_{A \in \mathcal{B}}\big|\p\big(\wt{\bm{S}}_n \in A\big)-\p\big(\wt{\bm{Z}} \in A\big)\big| \le d_\CD\big(\wt{\bm{S}}_n,\wt{\bm{Z}}\big)$, for all $n \in \N$, where $\wt{\bm{Z}}=(X,Y)^\top \sim \mathcal{N}_2(\bm{0},\bm{I}_2)$ such that $\|\wt{\bm{Z}}\|^2=X^2+Y^2\eqd R \sim \chi^2(2)$. We will now show that the quantity $d_\CD(\wt{\bm{S}}_n,\wt{\bm{Z}})$ can be bounded using the result from Theorem~\ref{thm:iid_MOM} by rewriting $\wt{\bm{S}}_n$ in terms of Hermite polynomials. Note for all $x \in \R$, that
\begin{equation}\label{eq:skew_kurt_intermsof_H_3H_4}
   \begin{pmatrix}
       x^3/\sqrt{15} \\ (x^4-3)/\sqrt{96}
   \end{pmatrix}
   = \begin{pmatrix}
       (H_3(x)+3H_1(x))/\sqrt{15} \\ (H_4(x)+6H_2(x))/\sqrt{96}
   \end{pmatrix} 
   = \begin{pmatrix}
       \sqrt{\frac{9}{15}} & 0 & \sqrt{\frac{6}{15}} & 0 \\ 0 & \frac{\sqrt{3}}{2} 
       & 0 & \frac{1}{2}
   \end{pmatrix} \begin{pmatrix}
        H_1(x) \\ H_2(x)/\sqrt{2!} \\ H_3(x)/\sqrt{3!} \\ H_4(x)/\sqrt{4!}
   \end{pmatrix}.
\end{equation}
Define now the ensuing invertible matrix:
\begin{equation}
    \bm{\Delta}= \begin{pmatrix}
       \sqrt{\frac{9}{15}} & 0 & \sqrt{\frac{6}{15}} & 0 \\ 0 & \frac{\sqrt{3}}{2} 
       & 0 & \frac{1}{2}
       \\
       0 & 0 & 1 & 0 \\
       0 & 0 & 0 & 1
   \end{pmatrix}, \quad \text{ where } \bm{\Delta}\bm{\Delta}^\top =\begin{pmatrix}
       1 & 0 & \sqrt{\frac{6}{15}} & 0 \\ 0 & 1 & 0 & \frac{1}{2} \\ \sqrt{\frac{6}{15}} & 0 & 1 & 0 \\ 0 & \frac{1}{2} & 0 & 1
   \end{pmatrix}.
\end{equation} Let $\bm{S}_n$ be as in~\eqref{eq:MOM_defn_S_n} with $d=4$, then it follows by~\eqref{eq:skew_kurt_intermsof_H_3H_4} and the definition of $d_\CD$, that 
\begin{equation}\label{eq:bound_by_extension}
    d_\mathcal{K} \big(R_n^{\mathrm{SB}},R\big) \le d_\CD\big(\wt{\bm{S}}_n,\wt{\bm{Z}}\big) \le d_\CD(\bm{\Delta} \bm{S}_n,\bm{Z})=d_\CD(\bm{S}_n,\bm{\Delta}^{-1}\bm{Z}),
\end{equation} where $\bm{Z} \sim \mathcal{N}_4(\bm{0},\bm{\Sigma})$, with $ \bm{\Sigma}= \bm{\Delta}\bm{\Delta}^\top$. Hence, since $\bm{\Delta}^{-1}\bm{Z} \sim \mathcal{N}_4(\bm{0},\bm{I}_4)$, Theorem~\ref{thm:iid_MOM} together with Remark~\ref{rem:alt_up_bound_d} yields $d_\mathcal{K}\big(R_n^{\mathrm{SB}},R\big) \le d_\CD(\bm{S}_n,\bm{\Delta}^{-1}\bm{Z}) \le C_4n^{-1/2}$, for all $n \in \N$, where $C_4=4 (8 + 21 \sqrt{2}) (3^{3/4} + 15^{3/4} + 3 \cdot 71^{3/4} \sqrt{7}+ 93^{3/4}) \approx 923.44$ is calculated as in~\eqref{eq:exact_up_bound_small_d}.

Part~(b). Due to the definition of the Hermite polynomials, we note for all $x \in \R$, that
\begin{gather*}
       \begin{pmatrix}
       x^3/\sqrt{15} \\ (x^4-3)/\sqrt{96} \\ (x^5-10x^3)/\sqrt{345} \\ (x^6-15x^4+30)/\sqrt{4770}
   \end{pmatrix}
   = \begin{pmatrix}
       (H_3(x)+3H_1(x))/\sqrt{15} \\ (H_4(x)+6H_2(x))/\sqrt{96} \\ (H_5(x)-15H_1(x))/\sqrt{345} \\ (H_6(x)-45H_2(x))/\sqrt{4770}
   \end{pmatrix} 
   = \wt{\bm{\Delta}} \begin{pmatrix}
        H_1(x) \\ H_2(x)/\sqrt{2!} \\ H_3(x)/\sqrt{3!} \\ H_4(x)/\sqrt{4!} \\ H_5(x)/\sqrt{5!} \\ H_6(x)/\sqrt{6!}
   \end{pmatrix}, \text{ where }\\
    \wt{\bm{\Delta}}=\begin{pmatrix}
       \sqrt{\frac{9}{15}} & 0 & \sqrt{\frac{6}{15}} & 0 & 0 & 0 \\ 
       0 & \frac{\sqrt{3}}{2} 
       & 0 & \frac{1}{2}
       & 0 & 0 \\
       -\frac{15}{\sqrt{345}} & 0 & 0 & 0 & 2 \sqrt{\frac{2}{23}} & 0 \\
       0 & -\frac{45\sqrt{2}}{\sqrt{4770}} & 0 & 0 & 0 & 2 \sqrt{\frac{2}{53}}
   \end{pmatrix}.
\end{gather*}
Recall that $\wt{\bm{Z}} \sim \mathcal{N}_4(\bm{0},\wt{\bm{\Sigma}})$, where $\wt{\bm{\Sigma}}=\wt{\bm{\Delta}}\wt{\bm{\Delta}}^\top $ with $\wt{\bm{\Sigma}}$ given in~\eqref{eq:HM4_cov}. Hence, as Remark~\ref{rem:natural_d_C_conv} and~\eqref{eq:bound_by_extension}, we can now deduce that 
\begin{equation*}
    d_\mathcal{K}\big(R_n^{\mathrm{HM4}},\|\wt{\bm{Z}}\|^2\big) \le 
    d_\CD\big(\wt{\bm{\Delta}}\bm{S}_n, \wt{\bm{Z}}\big)
    \le d_\CD(\bm{\Delta}\bm{S}_n,\bm{Z}) = d_\CD(\bm{S}_n,\bm{\Delta}^{-1}\bm{Z}), \quad \text{ for all }n \in \N,
\end{equation*} where $\bm{Z} \sim \mathcal{N}_6(\bm{0},\bm{\Sigma})$, with $\bm{\Sigma}= \bm{\Delta}\bm{\Delta}^\top $, and $\bm{\Delta}$ is the invertible matrix defined by
\begin{equation*}
    \bm{\Delta}= \begin{pmatrix}
       \sqrt{\frac{9}{15}} & 0 & \sqrt{\frac{6}{15}} & 0 & 0 & 0 \\ 
       0 & \frac{\sqrt{3}}{2} 
       & 0 & \frac{1}{2}
       & 0 & 0 \\
       -\frac{15}{\sqrt{345}} & 0 & 0 & 0 & 2 \sqrt{\frac{2}{23}} & 0 \\
       0 & -\frac{45\sqrt{2}}{\sqrt{4770}} & 0 & 0 & 0 & 2 \sqrt{\frac{2}{53}} \\
       0 & 0 & 0 & 0 & 1 & 0 \\
       0 & 0 & 0 & 0 & 0 & 1
   \end{pmatrix}.
\end{equation*}
Thus, since $\bm{\Delta}^{-1}\bm{Z} \sim \mathcal{N}_6(\bm{0},\bm{I}_6)$, it follows by Theorem~\ref{thm:iid_MOM} and Remark~\ref{rem:alt_up_bound_d}, that $d_\mathcal{K}\big(R_n^{\mathrm{HM4}},\|\wt{\bm{Z}}\|\big) \le d_\CD(\bm{S}_n,\bm{\Delta}^{-1}\bm{Z}) \le C_6 n^{-1/2}$, for all $n \in \N$, where $C_6 = 2 \sqrt{6} (8 + 21 6^{1/4}) (3^{3/4} + 15^{3/4} + 3\cdot 71^{3/4}\sqrt{7} + 93^{3/4} + 3 \cdot 517^{3/4} \sqrt{3} + 35169^{3/4}) \approx 673794.4769$, calculated by~\eqref{eq:exact_up_bound_small_d}.
\end{proof}

\section*{Acknowledgements}

\thanks{
\noindent ABO and DKB are supported by AUFF NOVA grant AUFF-E-2022-9-39. DKB would like to thank the Isaac Newton Institute for Mathematical Sciences, Cambridge, for support and hospitality during the programme Stochastic systems for anomalous diffusion, where work on this paper was undertaken. This work was supported by EPSRC grant EP/Z000580/1.}

\printbibliography

@article {MR2144310,
    AUTHOR = {Bentkus, V.},
     TITLE = {A {L}yapunov type bound in {${\bf R}^d$}},
   JOURNAL = {Teor. Veroyatn. Primen.},
  FJOURNAL = {Teoriya Veroyatnoste\u{\i} i ee Primeneniya},
    VOLUME = {49},
      YEAR = {2004},
    NUMBER = {2},
     PAGES = {400--410},
      ISSN = {0040-361X},
   MRCLASS = {60E15},
  MRNUMBER = {2144310},
MRREVIEWER = {Ludger R\"{u}schendorf},
       DOI = {10.1137/S0040585X97981123},
       URL = {https://doi.org/10.1137/S0040585X97981123},
}

@article {MR980849,
    AUTHOR = {Baringhaus, L. and Henze, N.},
     TITLE = {A consistent test for multivariate normality based on the
              empirical characteristic function},
   JOURNAL = {Metrika},
  FJOURNAL = {Metrika. International Journal for Theoretical and Applied
              Statistics},
    VOLUME = {35},
      YEAR = {1988},
    NUMBER = {6},
     PAGES = {339--348},
      ISSN = {0026-1335},
   MRCLASS = {62H15},
  MRNUMBER = {980849},
MRREVIEWER = {Rashid Ahmad},
       DOI = {10.1007/BF02613322},
       URL = {https://doi.org/10.1007/BF02613322},
}

@article {MR725389,
    AUTHOR = {Epps, T. W. and Pulley, Lawrence B.},
     TITLE = {A test for normality based on the empirical characteristic
              function},
   JOURNAL = {Biometrika},
  FJOURNAL = {Biometrika},
    VOLUME = {70},
      YEAR = {1983},
    NUMBER = {3},
     PAGES = {723--726},
      ISSN = {0006-3444},
   MRCLASS = {62F03},
  MRNUMBER = {725389},
       DOI = {10.1093/biomet/70.3.723},
       URL = {https://doi.org/10.1093/biomet/70.3.723},
}

@ARTICLE{Kol-33,
  AUTHOR       = {Kolmogorov, A.},
  TITLE        = {Sulla Determinazione Empirica di una Legge di Distribuzione [On the empirical determination of a distribution law]},
  YEAR         = {1933},
  JOURNAL      = {Giorn. Ist. Ital. Attuar.},
  volume       = {4},
  number       = {1},
  pages        = {83-91},
}

@article {MR1483,
    AUTHOR = {Smirnoff, N.},
     TITLE = {Sur les \'{e}carts de la courbe de distribution empirique},
   JOURNAL = {Rec. Math. N.S. [Mat. Sbornik]},
  FJOURNAL = {Rec. Math. N.S. [Mat. Sbornik]},
    VOLUME = {6(48)},
      YEAR = {1939},
     PAGES = {3--26},
   MRCLASS = {60.0X},
  MRNUMBER = {1483},
MRREVIEWER = {M. Kac},
}

@article {MR4003566,
    AUTHOR = {Rai\v{c}, Martin},
     TITLE = {A multivariate {B}erry-{E}sseen theorem with explicit
              constants},
   JOURNAL = {Bernoulli},
  FJOURNAL = {Bernoulli. Official Journal of the Bernoulli Society for
              Mathematical Statistics and Probability},
    VOLUME = {25},
      YEAR = {2019},
    NUMBER = {4A},
     PAGES = {2824--2853},
      ISSN = {1350-7265},
   MRCLASS = {60F05},
  MRNUMBER = {4003566},
MRREVIEWER = {N. C. Weber},
       DOI = {10.3150/18-BEJ1072},
       URL = {https://doi.org/10.3150/18-BEJ1072},
}

@article {MR69328,
    AUTHOR = {Robbins, Herbert},
     TITLE = {A remark on {S}tirling's formula},
   JOURNAL = {Amer. Math. Monthly},
  FJOURNAL = {American Mathematical Monthly},
    VOLUME = {62},
      YEAR = {1955},
     PAGES = {26--29},
      ISSN = {0002-9890},
   MRCLASS = {33.0X},
  MRNUMBER = {69328},
MRREVIEWER = {S. C. van Veen},
       DOI = {10.2307/2308012},
       URL = {https://doi.org/10.2307/2308012},
}

@article {MR1691731,
    AUTHOR = {Pluci\'{n}ska, Agnieszka},
     TITLE = {Some properties of polynomial-normal distributions associated
              with {H}ermite polynomials},
   JOURNAL = {Demonstratio Math.},
  FJOURNAL = {Demonstratio Mathematica},
    VOLUME = {32},
      YEAR = {1999},
    NUMBER = {1},
     PAGES = {195--206},
      ISSN = {0420-1213},
   MRCLASS = {60E10 (33C45)},
  MRNUMBER = {1691731},
}

@article{Lo2015HighMoments,
  author    = {Gane Samb Lo and Oumar Thiam and Mohamed Cheikh Haidara},
  title     = {High Moments {J}arque--{B}era Tests for Arbitrary Distribution Functions},
  journal   = {Applied Mathematics},
  volume    = {6},
  pages     = {707--716},
  year      = {2015},
  doi       = {10.4236/am.2015.64066},
  url       = {https://doi.org/10.4236/am.2015.64066}
}

@InProceedings{Max_Welling,
  title = 	 {Robust Higher Order Statistics},
  author =       {Welling, Max},
  booktitle = 	 {Proceedings of the Tenth International Workshop on Artificial Intelligence and Statistics},
  pages = 	 {405--412},
  year = 	 {2005},
  editor = 	 {Cowell, Robert G. and Ghahramani, Zoubin},
  volume = 	 {R5},
  series = 	 {Proceedings of Machine Learning Research},
  month = 	 {06--08 Jan},
  publisher =    {PMLR},
  url = 	 {https://proceedings.mlr.press/r5/welling05c.html},
  note =         {Reissued by PMLR on 30 March 2021.}
}

@article{DeClerk2022HigherOrder,
  author    = {Luke De Clerk and Sergey Savel'ev},
  title     = {An Investigation of Higher Order Moments of Empirical Financial Data and Their Implications to Risk},
  journal   = {Heliyon},
  volume    = {8},
  number    = {2},
  year      = {2022},
  issn      = {2405-8440},
  publisher = {Elsevier},
  doi       = {10.1016/j.heliyon.2022.e08833},
  url       = {https://doi.org/10.1016/j.heliyon.2022.e08833}
}

@book {MR2962301,
    AUTHOR = {Nourdin, Ivan and Peccati, Giovanni},
     TITLE = {Normal approximations with {M}alliavin calculus},
    SERIES = {Cambridge Tracts in Mathematics},
    VOLUME = {192},
      NOTE = {From Stein's method to universality},
 PUBLISHER = {Cambridge University Press, Cambridge},
      YEAR = {2012},
     PAGES = {xiv+239},
      ISBN = {978-1-107-01777-1},
   MRCLASS = {60H07 (60F05 60G15)},
  MRNUMBER = {2962301},
MRREVIEWER = {David Nualart},
       DOI = {10.1017/CBO9781139084659},
       URL = {https://doi.org/10.1017/CBO9781139084659},
}

@book {MR2200233,
    AUTHOR = {Nualart, David},
     TITLE = {The {M}alliavin calculus and related topics},
    SERIES = {Probability and its Applications (New York)},
   EDITION = {Second},
 PUBLISHER = {Springer-Verlag, Berlin},
      YEAR = {2006},
     PAGES = {xiv+382},
      ISBN = {978-3-540-28328-7; 3-540-28328-5},
   MRCLASS = {60-02 (60H07 60H30)},
  MRNUMBER = {2200233},
MRREVIEWER = {Daniel Ocone},
}

@article {MR2573557,
    AUTHOR = {Nourdin, Ivan and Peccati, Giovanni},
     TITLE = {Stein's method and exact {B}erry-{E}sseen asymptotics for
              functionals of {G}aussian fields},
   JOURNAL = {Ann. Probab.},
  FJOURNAL = {The Annals of Probability},
    VOLUME = {37},
      YEAR = {2009},
    NUMBER = {6},
     PAGES = {2231--2261},
      ISSN = {0091-1798},
   MRCLASS = {60F05 (60G15 60G60 60H05 60H07)},
  MRNUMBER = {2573557},
MRREVIEWER = {Jean-Pierre Fouque},
       DOI = {10.1214/09-AOP461},
       URL = {https://doi.org/10.1214/09-AOP461},
}

@misc{Basse-Bang_Quant,
      title={Quantative bounds for high-dimensional non-linear functionals of Gaussian processes}, 
      author={Andreas Basse-O'Connor and David Kramer-Bang},
      year={2025},
      eprint={2502.17718},
      archivePrefix={arXiv},
      primaryClass={math.PR},
      url={https://arxiv.org/abs/2502.17718}, 
}

@article {MR381079,
    AUTHOR = {Bowman, K. O. and Shenton, L. R.},
     TITLE = {Omnibus test contours for departures from normality based on
              {$\surd b\sb{1}$} and {$b\sb{2}$}},
   JOURNAL = {Biometrika},
  FJOURNAL = {Biometrika},
    VOLUME = {62},
      YEAR = {1975},
    NUMBER = {2},
     PAGES = {243--250},
      ISSN = {0006-3444},
   MRCLASS = {62E15},
  MRNUMBER = {381079},
MRREVIEWER = {N. L. Johnson},
       DOI = {10.1093/biomet/62.2.243},
       URL = {https://doi.org/10.1093/biomet/62.2.243},
}

@article{DECLERCQ1998101,
title = {Hermite normality tests},
journal = {Signal Processing},
volume = {69},
number = {2},
pages = {101-116},
year = {1998},
issn = {0165-1684},
doi = {https://doi.org/10.1016/S0165-1684(98)00093-0},
url = {https://www.sciencedirect.com/science/article/pii/S0165168498000930},
author = {David Declercq and Patrick Duvaut}
}

@article {MR2067685,
    AUTHOR = {Bontemps, Christian and Meddahi, Nour},
     TITLE = {Testing normality: a {GMM} approach},
   JOURNAL = {J. Econometrics},
  FJOURNAL = {Journal of Econometrics},
    VOLUME = {124},
      YEAR = {2005},
    NUMBER = {1},
     PAGES = {149--186},
      ISSN = {0304-4076},
   MRCLASS = {62F03},
  MRNUMBER = {2067685},
       DOI = {10.1016/j.jeconom.2004.02.014},
       URL = {https://doi.org/10.1016/j.jeconom.2004.02.014}
}

@article{AMENGUAL2024,
title = {Multivariate Hermite polynomials and information matrix tests},
journal = {Econometrics and Statistics},
year = {2024},
issn = {2452-3062},
doi = {https://doi.org/10.1016/j.ecosta.2024.01.005},
url = {https://www.sciencedirect.com/science/article/pii/S2452306224000054},
author = {Dante Amengual and Gabriele Fiorentini and Enrique Sentana}
}

@article {MR4814822,
    AUTHOR = {Poitras, Geoffrey},
     TITLE = {Parametric testing for normality against bimodal and unimodal
              alternatives using higher moments},
   JOURNAL = {Comm. Statist. Simulation Comput.},
  FJOURNAL = {Communications in Statistics. Simulation and Computation},
    VOLUME = {53},
      YEAR = {2024},
    NUMBER = {10},
     PAGES = {4771--4789},
      ISSN = {0361-0918},
   MRCLASS = {Expansion},
  MRNUMBER = {4814822},
       DOI = {10.1080/03610918.2022.2155313},
       URL = {https://doi.org/10.1080/03610918.2022.2155313},
}

\end{document}